\documentclass[11pt,reqno]{amsart}
\usepackage{hyperref}
\usepackage{soul}
\usepackage{ amssymb }
\theoremstyle{definition}
\newtheorem{thm}{Theorem}[section]
\newtheorem{pro}[thm]{Proposition}
\newtheorem{cor}[thm]{Corollary}
\newtheorem{lem}[thm]{Lemma}
\newtheorem{notation}[thm]{Notation}
\theoremstyle{definition}
\newtheorem{rem}[thm]{Remark}

\newtheorem{exam}[thm]{Example}

\numberwithin{equation}{section}
\usepackage{float}
\usepackage{pgf,tikz}

\newcommand \cochord{\operatorname{co-chord}}
\newcommand \reg{\operatorname{reg}}

\newcommand \ma{\operatorname{m}}

\newcommand \K{\mathbb{K}}
\newcommand \kk{\mathcal{K}}

\newcommand \kp{\kappa}

\begin{document}
\title[Upper bounds for the regularity of symbolic powers of
edge ideals]{Upper bounds for the regularity 
of symbolic powers of certain classes of edge ideals}

\author[Arvind Kumar]{Arvind Kumar$^{1, 3}$}
	\email{arvindkumar@cmi.ac.in}
	\thanks{$^1$ The author is partially supported by Sciences and Engineering Research
Board, India under the National Postdoctoral Fellowship (PDF/2020/001436).}
	\author[S. Selvaraja]{S. Selvaraja$^{2, 3}$}
	\email{selva.y2s@gmail.com, sselvaraja@cmi.ac.in}
	\thanks{$^2$ The author is partially supported by DST, Govt of India under
the DST-INSPIRE Faculty Scheme (DST/Inspire/04/2019/001353).}
\thanks{$^3$ Both the authors are partially supported by the Infosys Foundation}
	\address{Chennai Mathematical Institute, H1, SIPCOT IT Park, Siruseri, Kelambakkam, Chennai 603103, Tamil Nadu, India}

\begin{abstract}
Let $G$ be a finite simple graph and $I(G)$ denote the corresponding edge ideal in a polynomial ring over a field $\K$. In this paper, we obtain upper bounds for the Castelnuovo-Mumford regularity of symbolic powers of certain classes of edge ideals. We also prove that for several classes of graphs, the regularity of symbolic powers of their edge ideals coincides with that of their ordinary powers.
\end{abstract}

\subjclass[AMS Classification 2010.]{Primary: 13D02, 05E40}


 \maketitle

 \section{Introduction}
 
 Let $I$ be a homogeneous ideal of a polynomial ring 
$S=\K[x_1,\ldots,x_n]$.
Then for $r \geq 1$, the $r$-th symbolic power of $I$ is defined as 
$I^{(r)}= \bigcap\limits_{\mathfrak{p}\in\min (I)} I^rS_{\mathfrak{p}} \cap S$,
where $\min(I)$ is the set of minimal prime ideals of $I$.
A classical result of Zariski-Nagata says that 
the $r$-th symbolic power of an 
ideal consists of the elements that vanish up to order $r$ on the corresponding variety.
Besides being an interesting
subject in its own right, symbolic powers appears as auxiliary tools in several important results in commutative algebra. 
In general, finding the generators of symbolic powers of $I$ is a challenging task.
Many authors studied symbolic powers (see \cite{DDAGHN} for a survey in this direction).

 Ever since Bertram, Ein, and Lazarsfeld proved that if $I$ is the defining ideal of a smooth complex projective variety, then $\reg(I^r)$ is bounded by a linear function of $r$, where $\reg(-)$ denotes the Castelnuovo-Mumford regularity, the study of the regularity of powers of homogeneous ideals of a polynomial ring has been a central problem in commutative algebra and algebraic geometry. One important result in this direction was given by Cutkosky, Herzog, and Trung \cite{CHT}, and independently by Kodiyalam \cite{vijay}. They proved that if $I$ is a homogeneous ideal of $S$, then the regularity
of $I^r$ is asymptotically a linear function in $r$
 i.e., there exist non-negative integers $a$ and $b$ depending on $I$ such that
$\reg(I^r)=ar+b \text{ for all $r \gg 0$}.$

Catalisano, Trung and Valla \cite[Proposition 7]{CTV93}, proved that, if 
$I$ defines $2q+1$ points on a rational normal curve in $\mathbb{P}^q$, $q \geq 2$, then for all 
$r \geq 1$, $\reg(I^{(r)})=2r+1+\lfloor \frac{r-2}{q} \rfloor$. Hence 
the function $\reg(I^{(r)})$ is not eventually linear in general. Minh and Trung asked, in \cite{MT19}, that if $I$ is a squarefree 
monomial ideal, then is $\reg(I^{(r)})$
eventually linear? Recently, Le Xuan et al. \cite{DHNT19}, gave a 
counterexample to the above question. 
Let $G$ be a simple (no loops, no multiple edges) undirected graph on the vertex set $\{x_1,\ldots,x_n\}$ and $I(G)$ denote the ideal generated by $\{x_i x_j \mid \{x_i , x_j \} \in E (G) \}$.
Minh \cite[p.1]{GHOJ18} conjectured that for any graph $G$,
 $\reg(S/I(G)^r)=\reg(S/I(G)^{(r)})$ for all $r \geq 1$.
It is known that $\reg(I(G)^r) = 2r + b$ for some $b$ and $r \geq r_0$.
While the aim is to obtain the linear polynomial corresponding to $\reg(I(G)^r)$, it 
seems unlikely that a single combinatorial invariant will represent the constant term 
for all graphs.
Establishing a relationship between the regularity of powers of edge ideals and 
combinatorial invariants associated with graphs such as the induced matching number
and the co-chordal cover number has been an active topic of research in the past decade 
( cf. \cite{selvi_ha}, \cite{jayanthan}, \cite{JS18}).
It was proved in \cite{selvi_ha} and \cite{JS18} that for any graph $G$,
 \begin{equation}\label{up-lp-reg}
2r+\nu(G) - 2 \leq \reg(S/I(G)^{r}) \leq 2r+\cochord(G) - 2, 
\text{ for all } r \geq 1,\end{equation}
where $\nu(G)$ denotes the induced matching number of
$G$ and $\cochord(G)$ denotes the co-chordal cover number of $G$.

In \cite{SVV} Simis, Vasconcelos and Villarreal proved that
$G$ is a bipartite graph if and only if $I(G)^r = I(G)^{(r)}$ for all $r \geq 1$. 
Therefore, Minh's conjecture is trivially true in this case. If $G$ is not bipartite, then it contains an odd cycle. Therefore, the
first case of study to verify Minh's conjecture is the class of odd cycle graphs, and this has been already done by Gu et al., in \cite{GHOJ18}. 
They also proved that for any graph $G$, 
$2r+\nu(G)-2\leq \reg(S/I(G)^{(r)}), \text{ for all }r\geq 1.$
Jayanthan and Kumar \cite{JK19} proved that if $G$ is a clique sum of an odd cycle with certain bipartite graphs, 
then $\reg(S/I(G)^{(r)})=\reg(S/I(G)^r)$, for all $r\geq 1$. Recently, Fakhari \cite{Seyed-chordal} proved that if $G$ is a chordal graph, then $\reg(S/I(G)^{(r)})=\reg(S/I(G)^r)=2r+\nu(G)-2$ 
for all $r \geq 1$. He also proved that if $G$ is a unicyclic graph, then $\reg(S/I(G)^{(r)})=
\reg(S/I(G)^r)=2r+\reg(S/I(G))-2$, for all $r\geq 1$, \cite{fak-uni}. 
In \cite{KKS19}, Kumar, Kumar and Sarkar proved that 
if $G$ is either a complete multipartite graph or a wheel graph, then for all $r \geq 2$,
$\reg(S/I(G)^{(r)})=\reg(S/I(G)^r)=2r+\nu(G)-2$.
 
There is no general upper bound known for $\reg(S/I(G)^{(r)})$. 
Considering the conjectures of Alilooee, Banerjee, 
Beyarslan, H\'a \cite[Conjecture 7.11(2)]{BBH17} 
and Minh, one may ask the following questions:

\begin{enumerate}
 \item[Q1] Is $\reg(S/I(G)^{(r)}) \leq 2r+\reg(S/I(G))-2$ for all $r \geq 1$? 
 \item[Q2] Is $\reg(S/I(G)^{(r)}) $ a linear function for $r\gg0$? If yes, can one obtain the linear polynomial corresponding to
  $\reg(S/I(G)^{(r)})$?
\end{enumerate}
We shall address the above problems.

There has been a lot of work on algebraic and combinatorial properties of edge ideals/cover ideals of graphs attaching cliques. For example,
Villarreal proved that if $G$ is any graph, then $W(G)$ is a Cohen-Macaulay graph, where $W(G)$ is the graph obtained by adding a whisker to each vertex of $G$, \cite{vill_cohen}.
In \cite{DochEng09}, \cite{Wood2}, the authors 
showed that $W(G)$ is a vertex-decomposable graph. In \cite{BFH15}, Biermann et al., gave sufficient conditions on $S \subset V(G)$ such
that $G \cup W(S)$ is vertex-decomposable, where $G\cup W(S)$ is the graph obtained from $G$ by adding a whisker at each vertex in $S$ (see also \cite{FH}).
 Later, Hibi et al., \cite{Hibi_cameronwalker} gave a generalization of Villarreal's result by showing that the graph obtained by attaching a clique to each vertex of a graph $G$ is unmixed and vertex-decomposable.
In a different direction, several authors have studied similar phenomena (cf. \cite{Haj_yessemi, AFY15}). 
In this paper, we consider the graph obtained by attaching a connected graph to some of the vertices of a graph. 
Let $H$ be a graph and $T=\{x_{i_1},\ldots,x_{i_q}\} \subseteq V(H)$. 
 The graph 
 $H_T=H(\kk(x_{i_1}), \ldots, \kk(x_{i_q}))$ is 
 obtained from $H$ by attaching $\kk(x_{i_j})$ to $H$ at $x_{i_j}$ for all
 $1 \leq j \leq q$, where $\kk(x)$ is a graph obtained by attaching some 
 complete graphs at a common vertex $x$ (see Section \ref{upperbound} 
 for definition). 
 We prove:
 \vskip 2mm \noindent
\textbf{Theorem \ref{bipartite}.}
{\em Let $H$ be a bipartite graph and $G=H_T$ for some $T \subseteq V(H)$.
Then for all $r \geq 1$,}
\[
 \reg(S/I(G)^{(r)}) \leq 2r+\reg(S/I(G))-2.
 \]

 We then move on to compute precise expressions for the regularity of symbolic powers of edge
ideals. We observe that for certain classes of graphs, the induced matching number coincides with the co-chordal cover number; for example, 
Cameron-Walker graphs, a subclass of weakly chordal graphs and certain whiskered graphs (Proposition \ref{whis-bipartite}).
We then use the Corollary \ref{bipartite-cor} and \eqref{up-lp-reg}
for such classes of graphs to
get $\reg(S/I(G)^{(r)})=\reg(S/I(G)^r)=2r+\nu(G)-2$ for all $r \geq 1$ (Corollary \ref{cam-wal}).

The second main result of the paper answers the question Q1 for a more general class than that of unicyclic graphs. Specifically, we show that the
 upper bound given in Q1 is attained by this class of graphs; that is,
 \vskip 1mm \noindent
\textbf{Theorem \ref{main-unicyclic}.}
Let $H$ be a unicyclic graph and  $G=H_T$ for some $T \subseteq V(H)$. Then 
 for all $r \geq 1$,
 \[
  \reg(S/I(G)^{(r)})=\reg(S/I(G)^r)=2r+\reg(S/I(G))-2.
 \]
  
 Our paper is organized as follows. In Section \ref{pre}, we collect the necessary notation, terminology, and some results that are used in the rest of the paper. 
In Section \ref{upperbound}, we prove the upper bound for the regularity of symbolic powers of $I(G)$ when $G=H_T$ and $H$ is a bipartite graph.
Finally, we compute the precise expressions for the regularity of symbolic
powers of certain classes of edge ideals in Section \ref{equal}.

 \section{Preliminaries}\label{pre}
 Throughout this paper, $G$ denotes a finite simple graph. 
For a graph $G$, $V(G)$ and $E(G)$ represent the set of all
vertices and the set of all edges of $G$ respectively.
The \textit{degree} of a vertex $x \in V(G),$ denoted by 
$\deg_G(x),$ is the number of edges incident to $x.$
A subgraph $H \subseteq G$ is called \textit{induced} if for $u, v
\in V(H)$, $\{u,v\} \in
E(H)$ if and only if $\{u,v\} \in E(G)$.
For any vertex $u \in V(G)$, let $N_G(u) = \{v \in V (G)\mid \{u, v\} \in E(G)~\}$
and $N_G[u]= N_G(u) \cup \{u\}$. 
For $U \subseteq V(G)$, denote by $G \setminus U$ 
the induced subgraph of $G$ on the vertex set $V(G) \setminus U$.
For a subset $A \subseteq V(G)$, $G[A]$ denotes the induced subgraph of $G$
on the vertex set $A$.
 A subset
$X$ of $V(G)$ is called \textit{independent} if there is no edge $\{x,y\} \in E(G)$ 
for $x, y \in X$. A graph $G$ is called \textit{bipartite} if there are
two disjoint independent subsets $X, Y$ of $V(G)$ such that $X \cup Y
= V(G)$. 
Let $C_k$ denote the cycle on $k$ vertices.

 A \textit{matching} in a graph $G$ is a subgraph consisting of pairwise disjoint edges. 
 The \textit{matching number} of $G$, denoted by $\ma(G)$, is the maximum cardinality of a matching of $G$. If the subgraph is an induced subgraph, then the matching is an \textit{induced matching}. The largest size of an induced matching in $G$ is called its \textit{induced matching number} and denoted by $\nu(G)$.  
The \textit{complement} of a graph $G$, denoted by $G^c$, is the graph on the same
vertex set as $G$ in which $\{u,v\}$ is an edge of $G^c$ if and only if it is not an edge of $G$.
A graph $G$ is \textit{chordal} if every induced
cycle in $G$ has length $3$, and is co-chordal if $G^c$ is chordal.
The \textit{co-chordal cover number}, denoted by $\cochord(G)$, is the
minimum number $n$ such that there exist co-chordal subgraphs
$H_1,\ldots, H_n$ of $G$ with $E(G) = \bigcup_{i=1}^n E(H_i)$.
Observe that for any graph $G$, we have
\begin{equation}\label{inva-ineq}
 \nu(G) \leq \cochord(G) \leq \ma(G).
\end{equation} 
A graph is said to be a \textit{unicyclic} graph if it contains exactly one cycle as a subgraph. A {\it complete graph} is a graph in which each pair of vertices are adjacent.  A subset $U$ of $V(G)$ is said to be a \textit{clique} if the induced subgraph with vertex set $U$ is a complete graph. 
A \textit{simplicial vertex} of a graph $G$ is a vertex $x$ such that the neighbors of $x$ form a clique in
$G$.  
Note that if $\deg_G(x)\leq 1$, 
then $x$ is a simplicial vertex of $G$.  

Let $\kk(x)$ denote a graph obtained by attaching some complete graphs at a common vertex $x$. The graph $\kk(x)$ is said to be a
\textit{star graph} if every simplicial vertex has degree $1$ and it is said to be \textit{star complete} if there is a simplicial vertex of degree $ \geq 2$.

 Let $M$ be a graded module over standard graded polynomial ring 
$S = \K[x_1,\ldots,x_n]$. Let
the graded minimal free resolution of $M$ be 
\[
 0 \longrightarrow \bigoplus_{j \in \mathbb{Z}} S(-j)^{\beta_{p,j}(M)} 
\overset{\phi_{p}}{\longrightarrow} \cdots \overset{\phi_1}{\longrightarrow} 
\bigoplus_{j \in \mathbb{Z}} S(-j)^{\beta_{0,j}(M)} 
\overset{\phi_0}{\longrightarrow} M\longrightarrow 0,
 \]
where $p \leq n$ and
$\beta_{i,j}(M)\neq 0$ denote the 
$(i,j)$-th graded Betti number of $M$. 
The 
Castelnuovo-Mumford regularity of $M$, denoted by $\reg(M)$, is defined as 
$\reg(M)=\max \{j-i \mid \beta_{i,j}(M) \neq 0\}.$ 
Let $I$ be a non-zero proper homogeneous ideal of $S$. 
Then it is immediate from the definition that $\displaystyle\reg\left(S/I\right)=\reg(I)-1$.
If $I=S$, then we set $\reg(S/I)=-\infty$.

We use the following well-known results to prove an upper bound for the regularity
of symbolic powers 
of edge ideals inductively.
 \begin{lem}\label{Inequalities} \cite[Lemma 3.1]{Ha2}
If $H$ is an induced subgraph of $G,$ then $\reg(S/I(H)) \leq \reg(S/I(G)).$
\end{lem}

\begin{rem}\label{main-rmk}
	Let $I \subset S$ be a nonzero homogeneous ideal and $f \in S$ be a homogeneous polynomial of 
	degree $d>0$. If $(I:f)=S$, then $\reg(S/I:f)+d =-\infty$ and hence,  
	$\reg(S/I)=\reg(S/(I,f)) \leq \max \{\reg(S/(I:f))+d,\reg(S/(I,f))\}$. 
	If $(I:f)$ is a proper ideal, then by \cite[Lemma 1.2]{HTT}, $\reg(S/I) \leq \max \{\reg(S/(I:f))+d,\reg(S/(I,f))\}$. Hence, in 
	both cases, we have $\reg(S/I) \leq \max \{\reg(S/(I:f))+d,\reg(S/(I,f))\}$.
\end{rem}
  
\section{Upper bound}\label{upperbound}
In this section, we obtain an upper bound for the regularity of symbolic powers of edge ideals
of certain graphs.
We first prove a technical lemma which is used to prove our main results. 
For $U \subseteq V(G)$, set $x_U=\prod_{i \in U}x_i$.
\begin{lem}\label{tech-lemma}
	Let $G$ be a graph. 
	If $x_1 $ is a simplicial vertex of $G$ and $x_1 \in W \subseteq N_G[x_1]$, then for all $r \geq 2$, $\reg(S/I(G)^{(r)}) \leq$
	\begin{align*}
		 \max 
		\{\reg(S/I(G\setminus x_1)^{(r)}), 
		\reg(S /(I(G\setminus A)^{(r)}:x_B))+|B| \mid x_1 \in B, 
		A \cup B=W, A \cap B= \emptyset\}.\end{align*}
\end{lem}
\begin{proof}
	Observe that $(I(G)^{(r)},x_1)=(I(G\setminus x_1)^{(r)},x_1)$. It follows from 
	Remark \ref{main-rmk} that 
	$$\reg(S/I(G)^{(r)}) \leq \max \{\reg(S/I(G\setminus x_1)^{(r)}), \reg(S/(I(G)^{(r)}:x_1))+1\}.$$
	Now, to prove the assertion it is enough to prove that for any 
	$W \subseteq N_G[x_1]$ with $x_1 \in W$, 
	\[	\reg(S/(I(G)^{(r)}:x_1)) \leq  \max \{ \reg(S/(I(G\setminus A)^{(r)}:x_B))
	+|B|-1 \mid x_1 \in B, A \cup B=W, A \cap B= \emptyset\}.\]
	
	We prove this by induction on $|W|$. If $|W|=1$, then $W=\{x_1\}$. Therefore we are done. 
	Now, assume that $|W|>1$ and the result is true for $W' \subset N_G[x_1]$ with $x_1 \in W'$ 
	and $|W'|<|W|$. Without loss of generality, we may assume that $W = \{x_1,\ldots,x_l\}$. 
	Set $W'=\{x_1,\ldots,x_{l-1}\}$. By induction hypothesis, we have
	\begin{align}\label{tech1}
	& \reg(S/(I(G)^{(r)}:x_1)) \leq \nonumber \\ & \max \{ \reg(S/(I(G\setminus A)^{(r)}:x_B))+|B|-1 \mid x_1 \in B, A \cup B=W', A \cap B= \emptyset\}.
	\end{align}
Note that for every pair of $A,B$ such that $A\cup B= W'$, $A\cap B =\emptyset$ and $x_1 \in B$,
	$$((I(G\setminus A)^{(r)}:x_B),x_l)=((I(G\setminus (A\cup \{x_l\}))^{(r)}:x_B),x_l)\text{ and }$$ $$ 
	(I(G\setminus A)^{(r)}:x_Bx_l)=(I(G\setminus A)^{(r)}:x_{B\cup \{x_l\}}).$$ 
	Thus, by Remark \ref{main-rmk}, we get
	\begin{align*}\label{tech2}
		&\reg(S/(I(G\setminus A)^{(r)}:x_B))+|B|-1 \\&\leq \max\{
		\reg(S/(I(G\setminus A)^{(r)}:x_{B \cup \{x_l\}}))+1, \reg(S/(I(G\setminus (A\cup \{x_l\}))^{(r)}:x_B))\} +|B|-1 \\& =
		\max\{\reg(S/(I(G\setminus A)^{(r)}:x_{B \cup \{x_l\}}))+|B|, \reg(S/(I(G\setminus (A\cup \{x_l\}))^{(r)}:x_B))+|B|-1\} \\&=
		\max\{\reg(S/(I(G\setminus A)^{(r)}:x_{B \cup \{x_l\}}))+|B\cup x_l|-1, \reg(S/(I(G\setminus (A\cup \{x_l\}))^{(r)}:x_B))+|B|-1\}.
	\end{align*}
	Now, the assertion follows from equation \eqref{tech1} and above inequality.	
\end{proof}

For two disjoint graphs $G_1$ and $G_2$, we denote the union of $G_1$
and $G_2$ by $G_1 \coprod G_2$.

\begin{pro}\label{reg-sum}
 Let $G=G_1 \coprod G_2$ be a graph. For $r \geq 1$, if $\reg(S/I(G_i)^{(s)}) \leq 2s+\rho_i(G_i)-2$ for all $1 \leq s \leq r$, $1 \leq i \leq 2$, 
 then \[\reg(S/I(G)^{(s)}) \leq 2s+\rho_1(G_1)+\rho_2(G_2)-2 \text{ for all } 1 \leq s \leq r .\]
\end{pro}
\begin{proof}
It follows from \cite[Theorem 4.6(2)]{Ha_sumsym} that  
\begin{align*}
	& \reg(S/I(G)^{(s)})=\\& \max_{\begin{subarray}{l} n \in [1,s-1]\\ m \in [1,s] \end{subarray}} \Big\{ \reg(S/I(G_1)^{(s-n)})+\reg(S/I(G_2)^{(n)})+1, \reg(S/I(G_1)^{(s-m+1)})+\reg(S/I(G_2)^{(m)}) \Big\}. 
\end{align*}
Fix $1 \leq n\leq s-1 \leq r-1$ and $1 \leq m\leq s \leq r$. Then
\begin{align*} \reg(S/I(G_1)^{(s-n)}) +\reg(S/I(G_2)^{(n)})+1 & \leq  
 2(s-n)+\rho_1(G_1)-2 +2n +\rho_2(G_2)-2+1 \\ & \leq 
	2s+\rho_1(G_1) +\rho_2(G_2)-2
	\end{align*}
and 
\begin{align*} \reg(S/I(G_1)^{(s-m+1)})+\reg(S/I(G_2)^{(m)}) &\leq 2(s-m+1)+\rho_1(G_1)-2 +2m +\rho_2(G_2)-2\\
 &= 2s+\rho_1(G_1)+\rho_2(G_2) -2 .\end{align*}
Hence, for every $1 \leq s \leq r$, we have $\reg(S/I(G)^{(s)}) \leq 2s +\rho_1(G_1)+\rho_2(G_2)-2 .$
\end{proof}
We fix the notation, which we will use for the rest of the paper.

\begin{notation}\label{setup}
 Let $H$ be a graph and $T=\{x_{i_1},\ldots,x_{i_q}\} \subseteq V(H)$. 
 The graph 
 $$H_T=H(\kk(x_{i_1}), \ldots, \kk(x_{i_q}))$$ is 
obtained from $H$ by attaching $\kk(x_{i_j})$ to $H$ at $x_{i_j}$ for all
$1 \leq j \leq q$. Note that if $T=\emptyset$, then $H=H_T$. By reordering of the elements of $T$, throughout the paper we assume that
$\kk(x_{i_1}),\ldots,\kk(x_{i_p})$ are star graphs and 
$\kk(x_{i_{p+1}}),\ldots,\kk(x_{i_q})$ are star complete graphs for some $p \leq q$.
Set 
$\kappa(G)=\sum\limits_{j=1}^{q-p}|V(\kk(x_{i_{p+j}}))|$.
Note that if $p=q$, then $\kp(G)=0$ and  
if $\kp(G)>0$, then $\kp(G) \geq 3$. 

\begin{minipage}{\linewidth}
 \begin{minipage}{0.7\linewidth}
 	For example, let $H=C_5$ with $V(H)=\{x_1,\ldots,x_5\}$ and $T=\{x_1,x_3,x_4,x_5\}$.
 	Let $H_T$ be as given in the figure. Here $\kk(x_1),\kk(x_4)$ are star graphs and $\kk(x_3), \kk(x_5)$ are star complete graphs. Note that $|V(\kk(x_3))|=7$ and
 	$|V(\kk(x_5))|=5$. Therefore, $\kp(G)=12$.
 \end{minipage}
\begin{minipage}{0.27\linewidth}
\begin{figure}[H]

\begin{tikzpicture}[scale=0.45]
\draw (4,4)-- (3,3);
\draw (3,3)-- (3,1);
\draw (4,4)-- (5,3);
\draw (3,1)-- (5,1);
\draw (5,3)-- (5,1);
\draw (4,5)-- (4,4);
\draw (4,0)-- (5,1);
\draw (5,0)-- (5,1);
\draw (5,0)-- (4,0);
\draw (5,1)-- (5.7,1.6);
\draw (5,1)-- (5.8,0.34);
\draw (5.8,0.34)-- (7.42,0.3);
\draw (5.7,1.6)-- (7.46,1.54);
\draw (7.46,1.54)-- (7.42,0.3);
\draw (5.8,0.34)-- (5.7,1.6);
\draw (5.8,0.34)-- (7.46,1.54);
\draw (5.7,1.6)-- (7.42,0.3);
\draw (5,1)-- (7.46,1.54);
\draw (5,1)-- (7.42,0.3);
\draw (2,4)-- (3,3);
\draw (2.14,2.14)-- (1.78,3.26);
\draw (3,3)-- (2.14,2.14);
\draw (1.78,3.26)-- (3,3);
\draw (1,2.5)-- (1.78,3.26);
\draw (1,2.5)-- (2.14,2.14);
\draw (1,2.5)-- (3,3);
\draw (3,1)-- (2,1);
\draw (3,1)-- (2,0);
\draw (3,1)-- (3,0);
\begin{scriptsize}
\fill (4,4) circle (1.5pt);
\draw(4.08,3.5) node {$x_1$};
\fill (3,3) circle (1.5pt);
\draw(3.42,2.86) node {$x_5$};
\fill (3,1) circle (1.5pt);
\draw(3.42,1.26) node {$x_4$};
\fill (5,3) circle (1.5pt);
\draw (4.66,2.86) node {$x_2$};
\fill (5,1) circle (1.5pt);
\draw (4.66,1.26) node {$x_3$};
\fill (4,5) circle (1.5pt);
\draw (4.14,5.26) node {};
\draw (6.16,3.26) node {};
\draw (5.16,4.26) node {};
\fill (4,0) circle (1.5pt);
\draw (4.1,0.26) node {};
\fill (5,0) circle (1.5pt);
\draw (5.1,0.26) node {};
\fill (5.7,1.6) circle (1.5pt);
\draw (5.84,1.86) node {};
\fill (5.8,0.34) circle (1.5pt);
\draw (5.94,0.6) node {};
\fill (7.42,0.3) circle (1.5pt);
\draw (7.6,0.56) node {};
\fill (7.46,1.54) circle (1.5pt);
\draw (7.62,1.8) node {};
\fill (2,4) circle (1.5pt);
\draw (2.16,4.26) node {};
\fill (2.14,2.14) circle (1.5pt);
\draw (2.3,2.4) node {};
\fill (1.78,3.26) circle (1.5pt);
\draw (1.94,3.52) node {};
\fill (1,2.5) circle (1.5pt);
\draw (1.16,2.76) node {};
\fill (2,1) circle (1.5pt);
\draw (2.16,1.26) node {};
\fill (2,0) circle (1.5pt);
\draw (2.14,0.26) node {};
\fill (3,0) circle (1.5pt);
\draw(3.16,0.26) node {};
\end{scriptsize}
\end{tikzpicture}
\end{figure}
\end{minipage}
\end{minipage}
\end{notation}

In \cite{AN19}, 
Banerjee and Nevo proved that if 
 $H$ is a bipartite graph, then 
 \begin{equation}\label{BN:eq}
 \reg(S/I(H)^r) \leq 2r+\reg(S/I(H))-2 \text{ for all } r \geq 1.  
 \end{equation}

We now prove an upper bound for the regularity of symbolic powers of certain classes of edge ideals.

\begin{thm}\label{bipartite}
Let $H$ be a bipartite graph and $G=H_T$ for some $T \subseteq V(H)$.
Then for all $r \geq 1$,
\[
 \reg(S/I(G)^{(r)}) \leq 2r+\reg(S/I(G))-2.
 \] 
\end{thm}
\begin{proof}
Let $T=\{ x_{i_1},\ldots,x_{i_q}\}$ and $G=H(\kk(x_{i_1}), \ldots, \kk(x_{i_q}))$.  We prove the assertion by induction on $r+\kp(G)$.
 If $r=1$, then we are done. If $\kp(G)=0$, then $G$ is a bipartite graph.
 Hence by \eqref{BN:eq} and \cite[Theorem 5.9]{SVV}, 
 \[
 \reg(S/I(G)^{(r)})=\reg(S/I(G)^r) \leq 2r+\reg(S/I(G))-2 \text{ for all } r \geq 1.
 \]
 
 Assume that $r \geq 2$ and $\kp(G) \geq 3$. There exists a simplicial vertex $x_1 \in V(G) \setminus V(H)$ such that $\deg_G(x_1) \geq 2$.  
 Without loss of generality, we can assume that $x_1 \in \kk(x_{i_q})$ and
 $N_G[x_1]=\{x_1,\ldots,x_\ell=x_{i_q}\}$. Note that $G\setminus x_1=H(\kk(x_{i_1}),\cdots, \kk(x_{i_q})\setminus x_1)$ and $\kp(G\setminus x_1 )+1 \leq \kp(G)$. 
 Thus, by induction on $r+\kp(G)$ and Lemma \ref{Inequalities},
 \[
 \reg(S/I(G \setminus x_1)^{(r)}) \leq 2r+\reg(S/I(G \setminus x_1))-2 \leq 2r+\reg(S/I(G))-2.
 \]
 
By applying Lemma \ref{tech-lemma}, it is enough to prove that for every pair of $A$ and $B$ such that  
$x_1 \in B, \; A\cap B=\emptyset, \; A \cup B =N_G[x_1]$,
$ \reg( S/(I(G \setminus A)^{(r)}:x_B))+|B| \leq 2r+\reg(S/I(G))-2.$
Since $x_1$ is a simplicial vertex of $G \setminus A$ and 
$N_{G \setminus A}[x_1]=B$, by \cite[Lemma 2]{Seyed_us}, 
\begin{align*}\label{eq:main-co}
 (I(G \setminus A)^{(r)}:x_B)=I(G \setminus A)^{(r-|B|+1)}.
\end{align*}
If $|B| \geq r+1$, then $(I(G\setminus A)^{(r)}:x_B) = S$. By Remark \ref{main-rmk}, \[\reg(S/(I(G\setminus A)^{(r)}:x_B))+|B|=-\infty < 2r+\reg(S/I(G))-2.\]
Now, if $ |B|=1$, then $A=N_G(x_1)$ and $B=\{x_1\}.$ Therefore,  
$(I(G\setminus A)^{(r)}:x_1)=I(G\setminus A)^{(r)}$ and $\kappa(G\setminus A)<\kappa(G)$. Thus, 
$$\reg(S/I(G\setminus A)^{(r)}:x_1)+1 \leq 2r+\reg(S/I(G\setminus A))-2+1
\leq 2r+ \reg(S/I(G\setminus A))-1.
$$
Since $\deg_G(x_1)\geq 2$, we have $x_2 \in N_G(x_1)$ and 
$ (G \setminus A) \coprod \{x_1,x_2\}$ is an induced subgraph of $G$. 
Therefore, by \cite[Lemma 8]{russ}, $\reg(S/I(G\setminus A))+1 \leq \reg(S/I(G))$. Thus,  
$$\reg(S/I(G\setminus A)^{(r)})+ 1 \leq 2r+\reg(S/I(G))-2.$$

We now assume that $2 \leq |B| \leq r$. 
Suppose that $x_\ell \in A$. Let $G_1,\ldots, G_k$ be connected components of 
$G\setminus A$. Without loss of generality assume that $G_1=(H\setminus x_l)(\kk(x_{i_1}),\ldots,\kk(x_{i_{q-1}}))$. 
Therefore, $G_2,\ldots,G_k$ are cliques. Clearly, $H \setminus x_l$ is a bipartite 
graph. Thus $G\setminus A$ is disjoint union of $G_1$ and a chordal graph 
$G'=G_2 \coprod \cdots\coprod G_k$. Notice that $I(G\setminus A)=I(G_1)+I(G')$. 
Since, $\kappa(G_1)<\kappa(G)$, by induction for all $1 \leq s \leq r$, 
$\reg(S/I(G_1)^{(s)}) \leq 2s+\reg(S/I(G_1))-2$. By \cite[Theorem 3.3]{Seyed-chordal}, 
$\reg(S/I(G')^{(s)}) =2s+\nu(G')-2=2s+\reg(S/I(G'))-2$ for all $s \geq 1$. 
Along these lines, by Proposition \ref{reg-sum}, for every $1 \leq s \leq r$,  
we have
\[\reg(S/I(G\setminus A)^{(s)}) \leq 2s +\reg(S/I(G_1))+\reg(S/I(G'))-2 
= 2s+\reg(S/I(G\setminus A))-2 ,\]
where the last equality follows from \cite[Lemma 8]{russ}.

Since, $2 \leq |B| \leq r$,   
\begin{align*}\reg(S/(I(G\setminus A)^{(r)}:x_B))+|B|& = \reg(S/I(G\setminus A)^{(r-|B|+1)})+|B| 
\\& \leq 2(r-|B|+1)+\reg(S/I(G\setminus A))-2+|B|\\& \leq 2r+ \reg(S/I(G\setminus A))-|B| 
\leq 2r+\reg(S/I(G)) -2.
\end{align*}
Suppose that $x_\ell \notin A$. Then $G \setminus A=
H(\kk(x_{i_1}), \ldots, \kk(x_{i_q})\setminus A)$. 
Since $\kp(G \setminus A)<\kp(G)$, by induction on $r+\kp(G)$,  
\begin{align*}
	\reg(S/I(G\setminus A)^{(r)}:x_B)+|B| &=
\reg(S/I(G \setminus A)^{(r-|B|+1)})+|B|\\ 
& \leq 2(r-|B|+1)+\reg(S/I(G\setminus A))-2+|B| \\
& \leq 2r+\reg(S/I(G))-2.
\end{align*}
Hence, the desired result follows.
\end{proof}

The following example shows that the inequality given in Theorem \ref{bipartite} could be asymptotically strict.
\begin{exam}
  Let $H=C_8 \coprod C_{8} \coprod C_{10} \coprod C_{10}$
 and $G=H_T=C_8 \coprod C_8 \coprod C_{10} \coprod (C_{10}(\kk(x_1)))$, where 
  $\kk(x_1)$ is a complete graph $K_3$. 
  Note that $\cochord((C_{10}(\kk(x_1)))=4=\nu((C_{10}(\kk(x_1))).$ Therefore,
  by \eqref{up-lp-reg}, $\reg(S/I(C_{10}(\kk(x_1)))=4$. By \cite[Theorem 7.6.28]{sean_thesis},
  $\reg(S/I(C_8))=3$ and $\reg(S/I(C_{10}))=3$. Therefore, by \cite[Lemma 8]{russ},
  \[
  \reg(S/I(G))=3+3+3+4=13.
  \]
It follows from \cite[Corollary 5.4]{GHOJ18} that $\reg(S/I(C_8)^{(r)})=2r$ and $\reg(S/I(C_{10})^{(r)})=2r+1$
for all $r \geq 2$. By Theorem \ref{bipartite} and \cite[Theorem 4.6]{GHOJ18}, $\reg(S/I(C_{10}(\kk(x_1))^{(r)}))=2r+2$. Now, by \cite[Theorem 5.11]{Ha_sumsym}, for all $r \geq 2$,
\[
 \reg(S/I(G)^{(r)})=2r+9<2r+11.
\]
 \end{exam}

As an immediate consequence of Theorem \ref{bipartite}, we have the following statement.
\begin{cor}\label{bipartite-cor} Let $G$ be a graph as in Theorem \ref{bipartite}. 
Then for all $r \geq 1$,
 \[
 \reg(S/I(G)^{(r)}) \leq 2r+\cochord(G)-2.
 \] 
 
\end{cor}

\begin{proof}
 The assertion follows from Theorem \ref{bipartite} and \cite[Theorem 1]{russ}.
\end{proof}

\section{Precise expressions for asymptotic regularity}\label{equal}
In this section, 
we prove that for
several classes of graphs, the regularity of symbolic powers of their edge ideals coincides with
that of their ordinary powers.

A graph $G$ is \textit{weakly chordal} if every induced cycle in both $G$ and $G^c$ has length at most $4$. 
A weakly chordal graph that is also bipartite is called a \textit{weakly chordal bipartite graph}.
A graph $G$ satisfies $\nu(G)=\ma(G)$ is called a \textit{ Cameron-Walker graph.}
Cameron and Walker \cite[Theorem 1]{CamWalker} and Hibi et al.,\cite[p. 258]{Hibi_cameronwalker} gave a classification of the 
connected graphs with $\nu(G)=\ma(G)$.
A subset $C \subseteq V(G)$ is a vertex cover of
$G$ if for each $e \in E(G)$, $e\cap C \neq \phi$. If $C$ is minimal
with respect to inclusion, then $C$ is called \textit{minimal vertex
	cover} of $G$.
We now obtain a class of graphs for which the induced matching number equals the co-chordal cover number.  

\begin{pro}\label{whis-bipartite}
 If \begin{enumerate}
\item $G=H_T$, where $H$ is a weakly chordal graph and $T \subseteq V(H)$,
\item $G$ is a Cameron-Walker graph,~~~~~~or 
 \item $H$ is a bipartite graph, $G=H_T$ and 
$T$ is a vertex cover of $H$,
\end{enumerate}
then $\nu(G)=\cochord(G)$.  
\end{pro}
\begin{proof}
 (1) It is immediate from the definition of $H_T$ that if $H$ is a weakly chordal graph, then so is $H_T$ for 
	any $T \subseteq V(H)$. By \cite[Proposition 3]{busch_dragan_sritharan}, $\nu(G)=\cochord(G)$.
\vskip 1mm \noindent
(2) The assertion follows from \eqref{inva-ineq}.
\vskip 1mm \noindent
(3) Let $T=\{ x_{i_1},\ldots,x_{i_q}\}$ and $G=H(\kk(x_{i_1}), \ldots, \kk(x_{i_q}))$.
We have the following cases:	
\vskip 1mm \noindent
\textsc{Case a.} Suppose $\kk(x_{i_1}), \ldots, \kk(x_{i_q})$ are star 
graphs. 
\vskip 0.5mm \noindent
Let $G'$ be the induced subgraph of $G$ on the vertex set 
$\bigcup\limits_{j=1}^q V(\kk(x_{i_j}))$ and $H'=H[T]$. 
One can observe that $G'=H'_T$. Thus, it follows from \cite[Lemma 21]{russ} 
that $\cochord(G')=\nu(G')$. 
We claim that $\cochord(G)\leq \cochord(G')$. 
Let $\cochord(G')=r$ and $G_1',\ldots,G_r'$ be co-chordal 
subgraphs of $G'$ with $E(G')=\bigcup\limits_{k=1}^rE(G_k')$. 
For $ 1\leq k \leq r$, let $G_k$ be the graph obtained from $G_k'$ in the 
following way: $G_k=G_k'$, if $E(G_k') \cap E(\kk(x_{i_j})) =\emptyset$, 
for each $1 \leq j \leq q$, otherwise $G_k= G_k' \cup \{\{u,x_{i_j}\} \mid u \in 
N_H(x_{i_j}) \text{ and } E(G_k') \cap E(\kk(x_{i_j})) \neq \emptyset\}.$ 
It is clear that $G_1,\ldots,G_r$ are co-chordal graphs. 
Since $T$ is a vertex cover of $G$, $E(G)=\cup_{k=1}^rE(G_k)$. 
Therefore, $\cochord(G)\leq \cochord(G')$. Hence, $\cochord(G)\leq 
\cochord(G')=\nu(G') \leq \nu(G) \leq \cochord(G)$ which implies that 
$\cochord(G)=\nu(G)$.
\vskip 1mm \noindent
\textsc{Case b.} Suppose $\kk(x_{i_1}), \ldots, \kk(x_{i_q})$ are star complete graphs.
\vskip 1mm \noindent
Let $U_1,\ldots,U_{s}$ be the cliques of size $\geq 3$.
For each $1 \leq k \leq s$, let $e_k \in E(U_k)$ such that 
$e_k \cap \{x_{i_j}\} =\emptyset$, for each $1 \leq j \leq q$.  
Note that $\{e_k \mid 1 \leq k \leq s\}$ is an induced matching of $G$, 
therefore, $\nu(G) \geq s$. For $1 \leq k \leq s$, let $$G_k=G[U_k] \cup 
\{\{u,x_{i_j}\} | u \in N_H(x_{i_j}) \text{ and } x_{i_j}\in V(U_k) \} \cup 
\{\{u,x_{i_j}\} | u \in V(\kk(x_{i_j})), \deg_G(u)=1\}.$$ Clearly, $G_k$ is 
a co-chordal graph for each $1 \leq k \leq s$. Since, $T$ is a vertex cover of 
$G$, $E(G)=\cup_{k=1}^{s}E(G_k)$. Hence, we have $\cochord(G) \leq s \leq  
\nu(G) \leq \cochord(G)$ which implies that $\cochord(G)=\nu(G)=$ the number 
of cliques of size $\geq 3$.
\vskip 1mm \noindent
\textsc{Case c.} Suppose $\kk(x_{i_1}), \ldots, \kk(x_{i_p})$ are star graphs
and $\kk(x_{i_{p+1}}), \ldots, \kk(x_{i_q})$ are star complete graphs 
for some $ 1 \leq p <q$. 
\vskip 1mm 
Let $H_1$ be the induced subgraph of $G$ on the vertex set $U \cup \big
(\cup_{j=1}^p V(\kk(x_{i_j}))\big)$, where $U=\{u \in V(H)| \{u,x_{i_j}\} 
\in E(H) \text{ for some } 1 \leq j \leq p \}$ and $H_2$ be the induced 
subgraph
of $G$ without isolated vertices such that $E(H_2)=E(G) \setminus E(H_1)$.
Let $H_1'$ be the induced subgraph of $H$ on the vertex set 
$U \cup \{x_{i_1},\ldots,x_{i_p}\}$ and $H_2'$ be 
the induced subgraph of $H$ with $E(H_2')=E(H)\setminus E(H_1')$.
One can observe that $T_1 =\{x_{i_1},\ldots,x_{i_p}\}$ is a vertex cover 
of $H_1'$, $T_2=\{x_{i_{p+1}},\ldots,x_{i_q}\}$ is a vertex cover of $H_2'$, 
$H_1={(H_1')}_{T_1}$ and $H_2=(H_2')_{T_2}$.  Let $\{f_{l}\;|\; 1 \leq l \leq 
\nu(H_1)\}$ be an induced matching of $H_1$. By \textsc{Case b}, 
$\nu(H_2)=$ the number of cliques of $H_2$ of size $\geq 3$.  
Let $U_1,\ldots,U_{\nu(H_2)}$ be the cliques of $H_2$ of size $\geq 3$. For each $1 \leq k \leq \nu(H_2)$, let $e_k \in E(U_k)$ such that $e_k \cap 
\{x_{i_j}\} =\emptyset$, for each $p+1 \leq j \leq q$. Then $\{e_k\; |\; 1 \leq 
k \leq \nu(H_2)\}$ is an induced matching of $H_2$. Clearly, 
$\{f_{l}\;|\; 1 \leq l \leq \nu(H_1)\} \cup \{e_k \;|\; 1 \leq k \leq \nu(H_2)\}$ is 
an induced matching of $G$. Therefore, $\nu(H_1)+\nu(H_2) \leq \nu(G)$. 
By \textsc{Case a} and \textsc{Case b}, we have $\cochord(H_1)=\nu(H_1)$ 
and $\cochord(H_2)=\nu(H_2)$. Since, $E(G)=E(H_1)\sqcup E(H_2)$, 
$$\cochord(G) \leq \cochord(H_1)+\cochord(H_2) =\nu(H_1)+\nu(H_2) \leq \nu(G)
\leq \cochord(G).$$ Hence, $\cochord(G)=\nu(G)$. 
\end{proof}
As an immediate consequence of Proposition \ref{whis-bipartite}, we have the following:

\begin{cor}\label{cam-wal}
	If \begin{enumerate}
		\item $H$ is a weakly chordal bipartite graph and 
		$G=H_T$ for some $T \subseteq V(H)$,
		\item $G$ is a Cameron-Walker graph, ~~~~~or
		\item $H$ is a bipartite graph, $G=H_T$ and 
$T$ is a vertex cover of $H$,
	\end{enumerate}
	then for all $r \geq 1$,
	\[
	\reg(S/I(G)^{(r)})=\reg(S/I(G)^r)=2r+\nu(G)-2.
	\]
\end{cor}
\begin{proof}
The assertions follow from Proposition \ref{whis-bipartite}, \eqref{up-lp-reg}, \cite[Theorem 4.6]{GHOJ18} 
and Corollary \ref{bipartite-cor}.
\end{proof}
The result for the Cameron-Walker graph has been recently obtained in \cite{Seyed-CW}.
	
We now move on to prove the next main result of this section. First, we fix notations that are needed to prove our next main theorem.
\begin{notation}\label{nota-uni}
 Let $H$ be a unicyclic graph with cycle $C_n:y_1 \cdots y_n$. Let $G=H_T$ for some $T \subseteq V(H)$. Note that $G$ is obtained from $C_n$ by attaching chordal graphs say $G_1,\ldots, G_m$ at $y_{i_1},\ldots,y_{i_m}$ respectively, 
 where $\{y_{i_1},\ldots,y_{i_m}\} \subseteq V(C_n)$.
Set 
\[
 \Gamma(G)= \bigcup_{j=1}^m N_{G_i}(y_{i_j}) \text{ and } 
H_j \text{ to be induced subgraph of } G_j \text{ on the vertex set } V(G_j) \setminus \Gamma(G).
\]
Note that 
$G \setminus \Gamma(G) = C_n \coprod \Big(\coprod_{j=1}^{m}H_j\Big) \textrm{ and } \nu (G \setminus \Gamma(G) ) = \nu(C_n)+\sum_{j=1}^m \nu(H_j).$
\end{notation}

\begin{thm}\label{tech-uni} Let the notation be as in \ref{nota-uni}. Then
	$	\nu(G) \leq \reg(S/I(G)) \leq \nu(G)+1.$ Moreover,
 \begin{enumerate}
 \item If $n \equiv \{0,1\} (mod~3)$, then $\reg(S/I(G))=\nu(G)$.
 \item If $n \equiv 2~ (mod~3)$ and $\nu(G \setminus \Gamma(G))< \nu(G)$, 
 then $\reg(S/I(G))=\nu(G)$.
 \end{enumerate}
\end{thm}
\begin{proof}
Let $T=\{x_{i_1},\ldots,x_{i_q}\}$. By \eqref{up-lp-reg}, we have $\nu(G) \leq \reg(S/I(G))$. 
Therefore, it is enough to prove that $\reg(S/I(G)) \leq \nu(G)+1$. We prove this by induction on $q-p$.
If $p=q$, then $G$ is a unicyclic graph. 
Therefore, by \cite[Corollary 4.12]{BC2},
$\reg(S/I(G)) \leq \nu(G)+1$. Suppose $p<q$. It follows from Remark \ref{main-rmk} that
$$\reg(S/I(G)) \leq \max \{\reg(S/I(G \setminus x_{i_q})),\reg(S/I(G \setminus N_{G}[x_{i_q}]))+1\}.$$
If $H \setminus x_{i_q}$ is a forest, then $G\setminus x_{i_q}$ is a chordal 
graph. 
It follows from 
\cite[Corollary 6.9]{ha_adam} that $\reg(S/I(G\setminus x_{i_q})) 
\leq \nu(G\setminus x_{i_q})$. If $H \setminus x_{i_q}$ is a unicyclic graph, then  
$G\setminus x_{i_q}$ is disjoint union of $G_1=(H\setminus x_{i_q})(\kk(x_{i_1}), \ldots, \kk(x_{i_{q-1}})) $ and a chordal graph $G_2$.
By induction hypothesis, \cite[Lemma 8]{russ} and \cite[Corollary 6.9]{ha_adam}, 
$\reg(S/I(G \setminus x_{i_q})) \leq \nu(G\setminus x_{i_q})+1 \leq \nu(G)+1$. Similarly, $\reg(S/I(G \setminus N_{G}[x_{i_q}])))+1 
\leq 
\nu(G \setminus N_G[x_{i_q}])+2$. 
If $\{e_1,\ldots,e_t\}$ is an induced matching of 
$G \setminus N_G[x_{i_q}]$,
then $\{e_1,\ldots,e_t,\{x_1,x_2\}\}$ is an induced matching of $G$,
where $\{x_1,x_2\}$ is an edge of $\kk(x_{i_q})\setminus x_{i_q}$.
Therefore $\nu(G \setminus N_G[x_{i_q}])+1 \leq \nu(G)$. Hence
$\reg(S/I(G))\leq \nu(G)+1.$
 
For the second assertion,
it is enough to prove that $\reg(S/I(G)) \leq \nu(G)$.
\vskip 1mm \noindent
 (1) We prove the 
assertion by induction on $q-p$. If $p=q$, then $G$ is a unicyclic 
graph and the assertion follows from \cite[Lemma 3.3]{ABS19}. We now assume 
that $q>p$. The proof is in the same lines as the proof of the first assertion.
\vskip 1mm \noindent
(2) First, we claim that there exists a $k \in \{1, \ldots, m\}$ such that 
	$\nu(H_k) <\nu(G_k)$. Let $\mathcal{C}$ be an induced matching of $G$ such that $|\mathcal{C}|=\nu(G)$. One can decompose $\mathcal{C}$ as a union of an induced matching of $C_n$ and induced matchings of $G_j$'s. Hence  
	$$\nu(G) \leq \nu(C_n)+\sum_{j=1}^m \nu(G_j).$$ Now, 
	if $\nu(H_j)=\nu(G_j)$ for each $1 \leq j \leq m$, then  
	$\nu(G \setminus \Gamma(G))=\nu(G),$ which is not possible. Thus, 
	we have $k \in \{1, \ldots, m\}$ such that 
	$\nu(H_k) <\nu(G_k)$. 	
	 Observe that $G \setminus y_{i_k}$ and $G \setminus N_G[y_{i_k}]$ 
	 are chordal graphs. By \cite[Corollary 6.9]{ha_adam}, we have 
	$$\reg(S/I(G\setminus y_{i_k}))=\nu(G\setminus y_{i_k}) 
	\textrm{ and } \reg(S/I(G\setminus N_G[y_{i_k}]))=\nu(G\setminus 
	N_G[y_{i_k}]).$$
	
	 Let $H'$ be the induced subgraph of $G$ obtained by deleting  
	 $V(G_k) \cup N_G[y_{i_k}]$. Note that $G\setminus N_G[y_{i_j}] 
	 =H' \coprod H_k$ and therefore, $\nu(G\setminus N_G[y_{i_k}])
	 =\nu(H')+\nu(H_k)$. Since $\nu(H_k) <\nu(G_k),$ 
	 we have $\nu(G\setminus N_G[y_{i_k}])=\nu(H')+\nu(H_k)\leq \nu(H')+
	 \nu(G_k)-1 \leq \nu(G)-1$, where the last inequality follows as 
	 $H' \coprod G_k$ is an induced subgraph of $G$. Therefore 
	 $\reg(S/I(G\setminus N_G[y_{i_k}]))=\nu(G\setminus N_G[y_{i_k}]) 
	 \leq \nu(G)-1.$ By Remark \ref{main-rmk}, we have	
	 \[	\reg(S/I(G)) \leq \max \{ \reg(I(G\setminus y_{i_k})),
	 ~\reg(I(G \setminus N_G[y_{i_k}]))+1\} \leq \nu(G).
	 \] Hence, the assertion follows.
\end{proof}

We move on to give an upper bound for the regularity of powers of these edge ideals.
\begin{pro}\label{on-unicyclic}
	Let $H$ be a unicyclic graph and  $G=H_T$ for some $T \subseteq V(H)$. Then 
	for all $r \geq 1$,
	\[
	\reg(S/I(G)^r) \leq 2r+\reg(S/I(G))-2.
	\]
\end{pro}
\begin{proof}
	Let $G'$ be an induced subgraph of $G$. 
	Suppose
	there exists a vertex $x$ in $G'$ such that $\deg_{G'}(x)=1$.
	Since $G' \setminus N_{G'}[y] \coprod \{x,y\}$ is an induced
	subgraph of $G'$, where $y \in N_{G'}(x)$, it follows from \cite[Lemma 8]{russ}
	that $$\reg(S/I(G' \setminus N_{G'}[y]))+1 \leq \reg(S/I(G')).$$
	
	Suppose $\deg_{G'}(x)>1$ for all $x \in V(G')$.
	If $G'$ is a cycle, then one can see that there exists a vertex $z$ in $G'$
	such that $\reg(S/I(G'\setminus N_{G'}[z]))+1 \leq \reg(S/I(G'))$.
	If $G'=H'(\kk(z_{i_1}),\ldots,\kk(z_{i_s}))$, where $H'$ is an induced 
	subgraph of $H$ and 
	$\{z_{i_1},\ldots,z_{i_s}\} \subseteq T$, then
	$G' \setminus N_{G'}[z_{i_s}] \coprod \{z,z_{i_s}\}$ is an induced subgraph of $G'$, where
	$z \in V(\kk(z_{i_s}))$ and $z \neq z_{i_s}$. By \cite[Lemma 8]{russ},
	$$\reg(S/I(G' \setminus N_{G'}[z_{i_s}]))+1 \leq \reg(S/I(G')).$$
	
	Then it is easy to see that $\reg(-)$ satisfies (1)-(4) of \cite[Theorem 4.1]{JS18}.
	Hence, for all $ r \geq 1$, $\reg(S/I(G)^r) \leq 2r+\reg(S/I(G))-2.$
\end{proof}
We prove the same result for symbolic power as well.
\begin{pro}\label{odd-unicyclic}
Let $H$ be a unicyclic graph and  $G=H_T$ for some $T \subseteq V(H)$. Then \[\reg(S/I(G)^{(r)}) \leq 2r+\reg(S/I(G))-2, \text{ for all } r \geq 1.\]
\end{pro}
\begin{proof}
Let $T=\{ x_{i_1},\ldots,x_{i_q}\}$ and $G=H(\kk(x_{i_1}), \ldots, \kk(x_{i_q}))$.	
 We induct on $r+\kp(G)$.
 If $r=1$, then there is nothing to prove.  
 If $\kp(G)=0$, then $G$ is a unicyclic graph. Hence, the assertion follows from \cite[Theorem 3.9]{fak-uni}, \cite[Theorem 5.9]{SVV} and \cite[Theorem 5.4]{ABS19}.
 We consider that $r \geq 2$ and $\kp(G) \geq 3$. Clearly, $\kk(x_{i_q})$ is a star complete graph. 
 Let $x_1 \in V(\kk(x_{i_q})) \setminus \{x_{i_q}\}$ be a simplicial vertex such that $\deg_G(x_1) \geq 2$. Set $N_G[x_1]=\{x_1,\ldots,x_l\}$ and $x_l=x_{i_q}$. 
 Note that $G\setminus x_1=H(\kk(x_{i_1}), \ldots,\kk(x_{i_{q-1}}), \kk(x_{i_q})\setminus x_1)$ and $\kappa(G\setminus x_1)<\kappa(G)$. Thus, by induction,
 \[\reg(S/I(G\setminus x_1)^{(r)}) \leq 2r+\reg(S/I(G\setminus x_1))-2 \leq 2r+\reg(S/I(G))-2.\]
 Let $B \subseteq N_G[x_1]$ such that $x_1 \in B$. Set $A=N_G[x_1]\setminus B$. 
 Observe that $x_1$ is a simplicial vertex of $G\setminus A$ and $N_{G\setminus A}[x_1]=B$. 
 By \cite[Lemma 2]{Seyed_us}, $(I(G\setminus A)^{(r)}:x_B)=I(G\setminus A)^{(r-|B|+1)}$. 
 If $|B| \geq r+1$, then $(I(G\setminus A)^{(r)}:x_B) = S$. So, by Remark \ref{main-rmk}, \[\reg(S/(I(G\setminus A)^{(r)}:x_B))+|B|=-\infty < 2r+\reg(S/I(G))-2.\]
 Now, if $ |B|=1$, then $A=N_G(x_1)$ and $B=\{x_1\}.$ Therefore, 
 $(I(G\setminus A)^{(r)}:x_1)=I(G\setminus A)^{(r)}$. Thus,  
 $$\reg(S/I(G\setminus A)^{(r)}:x_1)+1 \leq 2r+\reg(S/I(G\setminus A))-2+1
 \leq 2r+ \reg(S/I(G\setminus A))-1.$$
 Since $\deg_G(x_1)\geq 2$, we have $|N_G[x_1]|=l \geq 3$.  
 Then $\reg(S/I(G\setminus A))+1 \leq \reg(S/I(G \setminus x_l)) \leq 
 \reg(S/I(G))$ which implies that 
 $\reg(S/I(G\setminus A)) \leq \reg(S/I(G))-1$. Therefore \[\reg(S/I(G\setminus A)^{(r)})+ 1 \leq 2r+\reg(S/I(G))-2.\]
Thus, we assume that $2 \leq |B| \leq r$.  
  We have following cases: 
 \vskip 1mm \noindent
 \textbf{Case a.} If $x_l \notin A$, then $G\setminus A=H(\kk(x_{i_1}), \ldots,\kk(x_{i_{q-1}}), \kk(x_{i_q})\setminus A)$.
 The induction argument yields that \begin{align*}\reg(S/(I(G\setminus A)^{(r)}:x_B))+|B|& 
 \leq 2(r-|B|+1)+\reg(S/I(G\setminus A))-2+|B|\\& \leq 2r+ \reg(S/I(G\setminus A))-|B|
 \leq 2r+\reg(S/I(G)) -2,
 \end{align*}
 where the last inequality follows from Lemma \ref{Inequalities}.
 \vskip 1mm \noindent
 \textbf{Case b.}
Now, if $x_l \in A$, then let $G_1,\ldots, G_k$ be connected components of 
$G\setminus A$. Without loss of generality assume that 
$G_1=(H\setminus x_l)(\kk(x_{i_1}),\ldots,\kk(x_{i_{q-1}}))$. Therefore,  
$G_2,\ldots,G_k$ are cliques. If $H \setminus x_l$ is a forest, then $G_1$ is a chordal 
graph and hence $G\setminus A$ is a chordal graph. 
It follows from 
\cite[Theorem 3.3]{Seyed-chordal} that for all $s \geq 1$, $\reg(S/I(G\setminus A)^{(s)}) 
= 2s+\reg(S/I(G\setminus A))-2$. If $H \setminus x_l$ is a unicyclic graph, then  
$G\setminus A$ is disjoint union of $G_1$ and a chordal graph 
$G'=G_2 \coprod \cdots\coprod G_k$. Since, $\kappa(G_1)<\kappa(G)$, by induction for all $1 \leq s \leq r$, $\reg(S/I(G_1)^{(s)}) \leq 2s+\reg(S/I(G_1))-2$. By \cite[Theorem 3.3]{Seyed-chordal}, $\reg(S/I(G')^{(s)}) =2s+\nu(G')-2=2s+\reg(S/I(G'))-2$ for all $s \geq 1$. Notice that $I(G\setminus A)=I(G_1)+I(G')$. 
By Proposition \ref{reg-sum}, for every $1 \leq s \leq r$,  \[\reg(S/I(G\setminus A)^{(s)})\leq 2s+\reg(S/I(G_1)) +\reg(S/I(G'))-2 = 2s +\reg(S/I(G\setminus A))-2,\] where the last equality follows from \cite[Lemma 8]{russ}.
 Then  
 \begin{align*}\reg(S/(I(G\setminus A)^{(r)}:x_B))+|B|& = 
 \reg(S/I(G\setminus A)^{(r-|B|+1)})+|B| \\& \leq 2(r-|B|+1)+\reg(S/I(G\setminus A))-2+
 |B|\\& \leq 2r+ \reg(S/I(G\setminus A))-|B|\leq 2r+\reg(S/I(G)) -2,
 \end{align*} as $2 \leq |B| \leq r$.
Hence, the assertion follows from Lemma \ref{tech-lemma}.
 \end{proof}
 Now we prove the last main theorem of this section.
 \begin{thm}\label{main-unicyclic}
 Let $H$ be a unicyclic graph and  $G=H_T$ for some $T(\neq \emptyset) \subseteq V(H)$. Then 
 for all $r \geq 1$,
 \[
  \reg(S/I(G)^{(r)})=\reg(S/I(G)^r)=2r+\reg(S/I(G))-2.
 \]
\end{thm}
\begin{proof}
Let $H$ be a unicyclic graph with cycle $C_n:y_1 \cdots y_n$.
Suppose $n \equiv \{0,1\}(mod~3)$. 
By Proposition \ref{odd-unicyclic}, 
Theorem \ref{tech-uni}(1), \cite[Theorem 4.6]{GHOJ18}, for all $r \geq 1$,
\[
 \reg(S/I(G)^{(r)})=2r+\nu(G)-2=2r+\reg(S/I(G))-2.
\]
It follows from Proposition \ref{on-unicyclic}, Theorem \ref{tech-uni}(1)
and \cite[Theorem 4.5]{selvi_ha} that for all $r \geq 1$,
\begin{align*}
 \reg(S/I(G)^{r})=2r+\nu(G)-2=2r+\reg(S/I(G))-2.
\end{align*}

We assume that $n \equiv 2(mod~3)$.
Suppose $\nu(G \setminus \Gamma(G))<\nu(G)$.
By Proposition \ref{odd-unicyclic}, 
Theorem \ref{tech-uni}(2), \cite[Theorem 4.6]{GHOJ18},
Proposition \ref{on-unicyclic}
and \cite[Theorem 4.5]{selvi_ha}, for all $r \geq 1$,
\[
 \reg(S/I(G)^{(r)})=\reg(S/I(G)^r)=2r+\nu(G)-2=2r+\reg(S/I(G))-2.
\]
Suppose $\nu(G \setminus \Gamma(G))=\nu(G)$.
Following the notation as in \ref{nota-uni},
 $G \setminus \Gamma(G)=C_n \coprod (\coprod_{j=1}^m H_j)$.
Set $I_1=I(C_n)$ and $I_2=I(\coprod_{j=1}^m H_j)$. Since $\coprod_{j=1}^m H_j$ is a chordal graph, by \cite[Theorem 3.3]{Seyed-chordal}, $\reg(S/I_2^{(r)})=\reg(S/I^r_2)=2r+\nu(\coprod_{j=1}^m H_j)-2$ for all
$r \geq 1$.
 By \cite[Corollary 5.4]{GHOJ18}, for all $r \geq 2$,
$\reg(S/I_1^{(r)})=2r+\nu(C_n)-2$. Therefore, by \cite[Theorem 5.11]{Ha_sumsym},
\[
 \reg(S/I(G\setminus \Gamma(G))^{(r)})=\reg(S/(I_1+I_2)^{(r)})=2r+\nu(G\setminus \Gamma(G))-1=2r+\nu(G)-1.
\]
It follows from \cite[Theorem 4.7, Theorem 5.2]{selvi_ha} and
\cite[Theorem 5.7]{nguyen_vu} that for all $r \geq 3$,
\begin{align*}
	\reg(S/I(G\setminus \Gamma(G))^r)
	&=2r+\nu(G \setminus \Gamma(G))-1=2r+\nu(G)-1.
\end{align*}
If $r=2$, then by \cite[Proposition 2.7 (ii)]{HTT}, 
 $\reg(S/I(G\setminus \Gamma(G))^2)
 =\nu(G \setminus \Gamma(G))+3=\nu(G)+3.$
Hence, by \cite[Corollary 4.3]{selvi_ha}, Proposition \ref{on-unicyclic} and 
\cite[Corollary 4.5]{GHOJ18}, Proposition \ref{odd-unicyclic}, 
  \[\reg(S/I(G)^{(r)})=\reg(S/I(G)^r)=2r+\reg(S/I(G))-2 \text{ for all }r \geq 1.\]
\end{proof}

\vskip 2mm
\noindent
\textbf{Acknowledgement:} We would like to thank A. V. Jayanthan for going through 
the manuscript and making some valuable suggestions. We also would like to thank Rajiv Kumar 
for valuable discussions.


\end{document}